\documentclass[12pt]{amsart}

\usepackage{amsmath}
\usepackage{amssymb}
\usepackage{amsthm}
\newtheorem{theorem}{Theorem}[section]
\newtheorem{lemma}[theorem]{Lemma}

\theoremstyle{definition}
\newtheorem{remark}[theorem]{Remark}

\newcommand{\tr}{\triangleright}

\begin{document}

\title{Endofunctors of Quandles and Racks}
\author{Peter Ulrickson}

\address{Department of Mathematics, The Catholic University of America, Washington, DC}
\email{ulrickson@cua.edu}

\maketitle

\begin{abstract}

 We show that the only endofunctors of the category of quandles commuting with the forgetful functor to sets are the power operations. We also give a similar statement for racks.
\end{abstract}

\section{Introduction}

Given a quandle $(X, \tr)$ the assignments $(X,\tr) \mapsto (X, (x,y) \mapsto x \tr^n y)$ for integers $n$ yield endofunctors of the category of quandles. Here $x \tr^n y$ denotes the element obtained through the $n$-fold application of $x \tr -$, i.e. the element $x \tr (x \tr (x \ldots (x \tr y ) \ldots ))$. In \cite{Szymik} it is noted that these functors, referred to as power operations, commute with the forgetful functor to sets and thus correspond to endofunctors of the algebraic theory of quandles. A question left open there is what other set-preserving endofunctors of the category of quandles there might be. In fact, these power operations are the only such endofunctors.

\begin{remark}
  Lawvere \cite{Lawvere} describes the correspondence between functors of categories of algebras for algebraic theories which preserve underlying sets and functors between algebraic theories. See also \cite{Freyd}, which refers to such underlying set preserving functors as \textit{Lawvere functors}. Thus we may speak of endofunctors of the algebraic theory of quandles while in fact working with endofunctors of the category of quandles (or racks) preserving underlying sets.
\end{remark}

\begin{remark}
  In \cite{PS} the power operations are treated in a more general setting. There, families of mutually distributive operations are considered. Converting from right to left distributivity, their Lemma 1 notes the requisite left self-distributivity of a composition of a left self-distributive operation with itself.
\end{remark}

Recall that any group $G$ has an associated quandle known as the conjugation quandle of $G$. The underlying set of this quandle is the underlying set of the group, endowed with the operation $x \tr y := xyx^{-1}$.

\begin{remark}
  There are two conventions for quandles and racks. The one we follow here is that of \cite{Szymik}, where $\tr$ is left self-distributive and each $x \tr -$ is bijective. Others, including \cite{PS}, consider operations $*$ which are right self-distributive and such that each $-*x$ is bijective.
\end{remark}

\section{Free Quandles and Conjugation Quandles of Free Groups}

There is a free functor from sets to quandles which is a left adjoint to the functor which forgets the binary operation. One construction of the free quandle on a given generating set is by combining elements of the set using the symbols $\tr$ and $\tr^{-1}$, considering such words modulo equivalence by the requisite relations for quandles. The symbol $\tr^{-1}$ is the binary operation satisfying $x \tr ( x \tr^{-1} y) = y$ for all $x$ and $y$ in the generating set. Note that self-distributivity of $\triangleright^{-1}$ is the same as that for $\triangleright$. Moreover, the two are mutually-distributive in the terminology of \cite{PS}, as in the following equation.
\begin{equation*}
x \triangleright (y \triangleright^{-1} z) = (x \triangleright y) \triangleright ^{-1}(x \triangleright z)
\end{equation*}

Elements in a free quandle can be put in a standard form. Given an element $q$ in the free quandle, and an element $x$ of the generating set, we refer to elements $x \tr q$ and $x \tr^{-1} q$ as left multiplications by $x$.

\begin{lemma} \label{leftMultiplications}
  Every element in a free quandle can be written as a composition of left multiplications by generating elements.
\end{lemma}

\begin{proof}
  First, observe the following equality. \[(x\triangleright y) \triangleright - = x \triangleright ( y \triangleright (x \triangleright^{-1} -))\] Repeating such a rewriting as necessary, it follows that multiplying  on the left by any composition of left multiplications is equivalent to applying a composition of left multiplications.

The case considered above shows that all elements of the free quandle involving only two operations can be written as a composition of left multiplications. By induction on the number of operations one sees that every element can be expressed as a composition of left multiplications by generating elements.
\end{proof}

We now relate endofunctors of the algebraic theory of quandles with elements of the free quandle on a two-element generating set. This was noted in Remark 5.5 of \cite{Szymik}.

\begin{lemma} \label{restriction}
  Endofunctors of the algebraic theory of quandles are determined by their restriction to the free quandle on a set with two elements.
\end{lemma}

\begin{proof}
  Once the new product $x * y$ is determined for the free quandle on two elements $x$ and $y$, the products on all quandles are determined, since for any pair of elements $a$ and $b$ in any quandle there is a unique quandle morphism extending $x \mapsto a$, $y \mapsto b$ and the functor must leave underlying set maps unchanged.
\end{proof}

Given an element $w$ in the free quandle on $\{x,y\}$ we now produce two elements in the free group on $\{x,y_1,y_2\}$. Let $g_i$ be the morphism from the free quandle on $\{x,y\}$ to the conjugation quandle of the free group on $\{x,y_1,y_2\}$ determined by the set map sending $x$ to $x$ and $y$ to $y_i$ for $i=1,2$. Then two group elements, denoted $w_1$ and $w_2$, are determined by $w_i := g_i(w)$.

\begin{lemma}
  If $w_1w_2^{-1}$ cyclically reduces to $y_1y_2^{-1}$, then $w_1$ is $x^kyx^{-k}$ for an integer $k$.
\end{lemma}
\begin{proof}
  The element $w_1$ is represented by a word with the form
  \begin{equation*}
x^{e_1}y_1^{f_1}x^{e_2}\ldots x^{e_k}y_1x^{-e_k}\ldots y_1^{-f_1}x^{-e_1}
\end{equation*}
and thus $w_1w_2^{-1}$ takes the form
\begin{equation*}
  x^{e_1}y_1^{f_1}x^{e_2}\ldots x^{e_k}y_1x^{-e_k}\ldots y_1^{-f_1}x^{-e_1}x^{e_1}y_2^{f_1}x^{e_2}\ldots x^{e_k}y_2^{-1}x^{-e_k}\ldots y_2^{-f_1}x^{-e_1}
\end{equation*}
which cyclically reduces to
\begin{equation*}
  y_1^{f_1}x^{e_2}\ldots x^{e_k}y_1x^{-e_k}\ldots y_1^{-f_1}y_2^{f_1}x^{e_2}\ldots x^{e_k}y_2^{-1}x^{-e_k}\ldots y_2^{-f_1}
\end{equation*}
If $f_1$ is non-zero this word is cyclically reduced, and hence cannot yield $y_1y_2^{-1}$. Thus $w_1$ must be of the form $x^ky_1x^{-k}$.
\end{proof}

We now determine the elements in the free quandle that lead to free group elements of the form considered in the previous construction.

\begin{lemma} \label{GroupToQuandle}
  Let $w$ be an element of the free quandle on $\{x,y\}$. If the image of $w$ under the morphism to the conjugation quandle of the free group on $\{x,y\}$ induced by the identity map on generating sets has the form $x^kyx^{-k}$ for an integer $k$, then $w = x \tr^ky$.
\end{lemma}
\begin{proof}
  The element $w$ may be written as a composition of left multiplications, as in Lemma \ref{leftMultiplications}, so that we have 
\begin{equation*}
w = z_1 \triangleright ^{e_1} (z_2 \triangleright ^{e_2} (z_3 \triangleright ^{e_3} (\ldots \triangleright (z_m \triangleright ^{e_m} y)))\ldots)
\end{equation*}
The exponents $e_i$ are either $1$ (yielding the ordinary operation) or $-1$ (yielding the inverse operation). The terms $z_i$ are either $x$ or $y$. The element in the free group on $x$ and $y$ produced from such a quandle element is represented by $(\Pi z_i^{e_i})y(\Pi z_i^{-e_i})$, and such a word is equivalent to $x^kyx^{-k}$ only if any extraneous initial occurences of $y$ are eliminated by adjacent occurrences of $y$ with exponent of opposite sign. 
\end{proof}

\section{Classification of Endofunctors}

 A one-relator group is a group with a presentation $\langle S|\{r\} \rangle $ where the set of relations consists of a single relation $r$. Magnus \cite{Magnus} showed the following, referred to as the Freiheitssatz, for one-relator groups.

\begin{theorem}[Freiheitssatz]
  Suppose that $\langle \{a_1,\ldots,a_n,x\}| \{r\} \rangle $ is a presentation of a group $G$ with a single relation $r$. If $x$ cannot be eliminated from $r$ through cyclic reduction, then $\{a_1,\ldots,a_n\}$ generates a free subgroup in $G$.
\end{theorem}

By constructing suitable one-relator groups we can use this theorem to exclude certain operations as possible endofunctors of the category of quandles.

\begin{theorem}
  The only endofunctors of the algebraic theory of quandles are the power operations.
\end{theorem}

\begin{proof}
  As was already noted in Lemma \ref{restriction}, such an endofunctor is determined by its restriction to the free quandle on $\{x,y\}$. Let $w$ be an element in this free quandle. Produce elements $w_1$ and $w_2$  in the free group on $\{x,y_1,y_2\}$ as above, and consider the one-relator group $G$ with presentation $\langle \{x,y_1,y_2\}|\{w_1w_2^{-1}\}\rangle$.

Suppose that a cyclically reduced expression for $w_1w_2^{-1}$ contains all three generators. By the Freiheitssatz the set $\{y_1,y_2\}$ generates a free subgroup, so that $y_1$ and $y_2$ are distinct. Use $w$ to produce a new binary operation $*$ on the (underlying set of the) conjugation quandle of $G$. This operation fails to define a quandle, since $x *y_1 = x*y_2$ even though $y_1 \ne y_2$. Thus $w$ does not yield an endofunctor of the algebraic theory of quandles.

If $w_1w_2^{-1}$ contains only $y_1$ and $y_2$ then it is equivalent to $y_1y_2^{-1}$, and Lemma \ref{GroupToQuandle} states that $w$ is of the form $x \tr ^k y$, which gives a power operation.
\end{proof}

The same reasoning which led to the previous theorem also allows for the classification of endofunctors of the algebraic theory of racks. Recall that a rack is a set with a binary operation $\tr$ which satisfies the same left self-distributivity condition as for quandles and for which $x \tr -$ is bijective. What is not required is that $x \tr x = x$. There is a forgetful functor from quandles to racks whereby we view the conjugation quandle of a group as a rack.

Just as in the case of quandles, the distributivity and bijectivity mean that one can write general elements of the free rack on two elements as a composition of left multiplications using $\tr$ and $\tr^{-1}$. A word $w$ in the free rack determining an endofunctor cannot contain terms of the form $y \tr x$ by the same Freiheitssatz argument. On the other hand, terms of the form $y \tr y$ cannot be reduced in a free rack, so that we have the following classification.

\begin{theorem}
  The only endofunctors of the algebraic theory of racks are of the form $x \tr^k F^j(y)$, where $F: y \mapsto y \tr y$ is the canonical automorphism of a rack and $k$ and $j$ are integers.
\end{theorem}

\begin{remark}
  As was observed in the remark of \cite{Szymik} raising the question addressed here, Kan \cite{Kan} showed that the algebraic theory of groups admits only the identity endofunctor.
\end{remark}

\vskip.1in

\noindent
\textbf{Acknowledgements:} I thank M. Szymik for his comments on a draft of this note.

\bibliographystyle{hplain}
\bibliography{endofunctorsOfQuandles.v2}

\end{document}